\documentclass[12pt, parskip=full]{scrartcl}

\usepackage[ngerman, english]{babel}

\usepackage[utf8]{inputenc}
\DeclareUnicodeCharacter{FB01}{fi}

\setkomafont{sectioning}{\normalfont\normalcolor\bfseries}

\usepackage{amsmath}
\usepackage{amsfonts}
\usepackage{amssymb}
\usepackage{amsthm}
\usepackage{stmaryrd}

\theoremstyle{plain}
\newtheorem{theorem}{Theorem}[section]
\newtheorem{lemma}[theorem]{Lemma}

\theoremstyle{definition}

\theoremstyle{remark}

\newtheorem*{example}{Example}
\newtheorem{remark}[theorem]{Remark}

\usepackage[numbers]{natbib}

\usepackage{url}

\usepackage{abstract}

\usepackage{changes}

\usepackage{hyperref}
\usepackage{graphicx}
\usepackage{xcolor}

\usepackage{pgf}
\usepackage{tikz}
\usetikzlibrary{arrows,automata,positioning,calc}

\title{Nonlinear Markov Chains with Finite State Space: Invariant Distributions and Long-Term Behaviour}
\author{Berenice Anne Neumann \thanks{University of Trier, Department IV, Universitätsring 19, 54296 Trier, Germany}}

\begin{document}
	
	\maketitle 
	
	\begin{abstract}
 		Nonlinear Markov chains with finite state space have been introduced in Kolokoltsov (2010) \cite{KolokoltsovNonLinearMakov}.
 		The characteristic property of these processes is that the transition probabilities do not only depend on the state, but also on the distribution of the process.
 		In this paper we provide first results regarding their invariant distributions and long-term behaviour.
 		We will show that under a continuity assumption an invariant distribution exists.
 		Moreover, we provide a sufficient criterion for the uniqueness of the invariant distribution that relies on the Brouwer degree.
 		Thereafter, we will present examples of peculiar limit behaviour that cannot occur for classical linear Markov chains.
 		Finally, we present for the case of small state spaces sufficient (and easy-to-verify) criteria for the ergodicity of the process.
	\end{abstract}
	
	\section{Introduction}
	
	Nonlinear Markov processes are a particular class of stochastic processes, where the transition probabilities do not only depend on the state, but also on the distribution of the process.
	McKean \cite{McKeanNonlinearMC} introduced these type of processes to tackle mechanical transport problems. Thereafter they have been studied by several authors (see the monographs of Kolokoltsov \cite{KolokoltsovNonLinearMakov} and Sznitman \cite{SznitmanNonlinearMC}). Recently, the close connection to continuous time mean field games led to significant progress in the analysis of McKean-Vlasov SDEs, in particular the control of these systems (see for example \cite{CarmonaMcKean,PhamWeiMcKeanMFG}).
	
	In this paper, we consider a special class of these processes, namely, nonlinear Markov chains in continuous time with a finite state space and provide first insights regarding the long-term behaviour of these processes. 
	Nonlinear Markov chains with finite state space arise naturally, in particular in evolutionary biology, epidemiology and game theory. 
	Namely, the replicator dynamics, several infection models, but also the dynamics of learning procedures in game theory are nonlinear Markov chains \cite{KolokoltsovNonLinearMakov}. Moreover, also the population's dynamics in mean field games with finite state and action space are nonlinear Markov chains \cite{NeumannDiss}.
	
	The main focus of this paper lies in the characterization of the long-term behaviour of these processes.
	We show that always an invariant distribution exists and provide a sufficient criterion for the uniqueness of this invariant distribution. 
	Thereafter, we turn to the long-term behaviour, where we first illustrate by two examples that the limit behaviour is much more complex than for classical Markov chains.
	More precisely, we show that the marginal distributions of a nonlinear Markov chain might be periodic and that irreducibility of the generator does not necessarily imply ergodicity.
	Then we provide easy-to-verify sufficient criteria for ergodicity for small state spaces (two or three states).
	All conditions that we propose are simple and rely only on the shape of the nonlinear generator, not on the shape of the transition probabilities.
	
	The long-term behaviour of nonlinear Markov chains in continuous time with a finite state space has not been analysed before.
	The closest contribution in the literature are ergodicity criteria for nonlinear Markov processes in discrete time \cite{ButkovskyNonlinearMC,SaburovNonLinearMC}. These criteria are a generalization of Dobrushin's ergodicity condition and the proofs crucially rely on the sequential nature of the problem.
	
	The rest of the paper is structured as follows: In Section \ref{Section:Notation} we review the relevant definitions and notation. In Section \ref{Section:InvariantDist} we present the results on existence and uniqueness of the invariant distribution. In Section \ref{Section:Examples} we provide examples of limit behaviour that cannot arise in the context of classical Markov chains. In Section \ref{Section:Ergodicity} we present the ergodicity results for small state spaces. The Appendix \ref{appendix} contains the proofs of two technical results.

	\section{Continuous Time Nonlinear Markov Chains with Finite State Space}	
	\label{Section:Notation}
	
	This section gives a short overview over the relevant definitions, notations and preliminary facts regarding nonlinear Markov chains. For more details regarding these processes we refer the reader to \cite[Chapter 1]{KolokoltsovNonLinearMakov}.
	Moreover, it introduces the relevant notions to characterize the long-term behaviour of these processes.
	
	Let $\mathcal{S}=\{1, \ldots, S\}$ be the state space of the nonlinear Markov chain and denote by $\mathcal{P}(\mathcal{S})$ the probability simplex over $\mathcal{S}$. 
	A nonlinear Markov chain is characterized by a \textit{continuous family of nonlinear transition probabilities} $P(t,m)= (P_{ij}(t,m))_{i,j \in \mathcal{S}}$ which is a family of stochastic matrices that depends continuously on $t \ge 0$ and $m \in \mathcal{P}(\mathcal{S})$ such that the nonlinear Chapman-Kolmogorov equation 
	$$\sum_{i \in \mathcal{S}} m_i P_{ij}(t+s,m) = \sum_{i,k\in \mathcal{S}} m_i P_{ik} (t,m) P_{kj} \left( s, \sum_{l \in \mathcal{S}} m_l P_l(t,m) \right)$$
	is satisfied. As usual $P_{ij}(t,m_0)$ is interpreted as the probability that the process is in state $j$ at time $t$ given that the initial state was $i$ and the initial distribution of the process was $m_0$. 
	Such a family yields a nonlinear Markov semigroup $(\Phi^t(\cdot))_{t \ge 0}$ of continuous transformations of $\mathcal{P}(\mathcal{S})$ via 
	$$\Phi^t_j(m) = \sum_{i \in \mathcal{S}} m_i P_{ij}(t,m) \quad \text{for all } t \ge 0, m \in \mathcal{P}(\mathcal{S}), j \in \mathcal{S}.$$
	Also $\Phi^t(m_0)$ has the usual interpretation that it represents the marginal distribution of the process at time $t$ when the initial distribution is $m_0$.
	A nonlinear Markov chain with initial distribution $m_0 \in \mathcal{P}(\mathcal{S})$ is then given as the time-inhomogeneous Markov chain with initial distribution $m_0$ and transition probabilities $p(s,i,t,j) = P_{ij}(t-s, \Phi^s(m_0))$.
	
	As in the theory of classical continuous time Markov chains the infinitesimal generator will be the cornerstone of the description and analysis of such processes:
	Let $\Phi^t(m)$ be differentiable in $t=0$ and $m \in \mathcal{P}(\mathcal{S})$, then the \textit{(nonlinear) infinitesimal generator} of the semigroup $(\Phi^t(\cdot))_{t \ge 0}$ is given by a transition rate matrix function $Q(\cdot)$ such that for $f(m):= \left. \frac{\partial}{\partial t} \Phi^t(m)\right|_{t=0}$ we have $f_j(m) = \sum_{i \in \mathcal{S}} m_i Q_{ij}(m)$ for all $j \in \mathcal{S}$ and $m \in \mathcal{P}(\mathcal{S})$.
	
	By \cite[Section 1.1]{KolokoltsovNonLinearMakov} any differentiable nonlinear semigroup has a nonlinear infinitesimal generator. 
	However, the converse problem is more important:
	Given a transition rate matrix function (that is a function $Q: \mathcal{P}(\mathcal{S}) \rightarrow \mathbb{R}^{S \times S}$ such that $Q(m)$ is a transition rate matrix for all $m \in \mathcal{P}(\mathcal{S})$) is there a nonlinear Markov semigroup (and thus a nonlinear Markov chain) such that $Q$ is the nonlinear infinitesimal generator of the process?
	Relying on the semigroup identity $\Phi^{t+s} = \Phi^t \Phi^s$ this problem is equivalent to the following Cauchy problem:
	Is there, for any $m_0 \in \mathcal{P}(\mathcal{S})$ a solution $(\Phi^t(m_0))_{t \ge 0}$ of 
	$$\frac{\partial}{\partial t} \Phi^t(m_0) = \Phi^t(m_0) Q( \Phi^t(m_0)), \quad \Phi^0(m_0)=m_0,$$
	such that $\Phi^t(\cdot)$ is a continuous function ranging from $\mathcal{P}(\mathcal{S})$ to itself and such that $\Phi^t(m) \in \mathcal{P}(\mathcal{S})$ for all $t \ge 0$ and $m \in \mathcal{P}(\mathcal{S})$.
	
	In the monograph \cite{KolokoltsovNonLinearMakov} the problem to construct a semigroup from a given generator is treated in a very general setting. Here, we present a result with easy-to-verify conditions tailored for the specific situation of nonlinear Markov chains with finite state space. The proof of the result, which relies on classical arguments from ODE theory, is presented in the appendix.
	
	\begin{theorem}
		\label{LipschitzContinuousGenerator}
		Let $Q: \mathcal{P}(\mathcal{S}) \rightarrow \mathbb{R}^{S \times S}$ be a transition rate matrix function such that $Q_{ij}(m)$ is Lipschitz continuous for all $i,j \in \mathcal{S}$.
		Then there is a unique Markov semigroup $(\Phi^t(\cdot))_{t \ge 0}$ such that $Q$ is the infinitesimal generator for $(\Phi^t(\cdot))_{t \ge 0}$.
	\end{theorem}  
	
	In this paper we are now mainly interested in the characterization of the long-term behaviour of nonlinear Markov chains:
	We say that $m \in \mathcal{P}(\mathcal{S})$ is an invariant distribution if $\frac{\partial}{\partial t} \Phi^0(m) = 0$ and thus also $\frac{\partial}{\partial t} \Phi^t(m) = 0$. An equivalent condition with respect to the generator is that a vector $m \in \mathcal{P}(\mathcal{S})$ is an invariant distribution if it solves $0 = m^TQ(m)$.

	We say that a nonlinear Markov chain with nonlinear semigroup $(\Phi^t(\cdot))_{t \ge 0}$ is \textit{strongly ergodic} if there exists an $\bar{m} \in \mathcal{P}(\mathcal{S})$ such that for all $m_0 \in \mathcal{P}(\mathcal{S})$ we have 
	$$\lim_{t \rightarrow \infty} \left| \left| \Phi^t(m_0) - \bar{m} \right| \right| = 0.$$	
	
	\section{Existence and Uniqueness of the Invariant Distribution}
	\label{Section:InvariantDist}
	
	The invariant distributions of a nonlinear Markov chain are exactly the fixed points of the set-valued map $$s: \mathcal{P}(\mathcal{S}) \rightarrow 2^{\mathcal{P}(\mathcal{S})}, \quad  m \mapsto \{ x \in \mathcal{P}(\mathcal{S}): 0 = x^TQ(m)\}.$$
	Using Kakutani's fixed point theorem, we directly obtain the existence of an invariant distribution for any generator:
	
	\begin{theorem}
		Let $Q(\cdot)$ be a nonlinear generator such that the map $Q: \mathcal{P}(\mathcal{S}) \rightarrow \mathbb{R}^{S\times S}$ is continuous. Then the nonlinear Markov chain with generator $Q(\cdot)$ has an invariant distribution.
	\end{theorem}
	
	\begin{proof}
		By \cite[Theorem 5.3]{IosifescuFiniteState} the set of all invariant distributions given a fixed generator matrix $Q(m)$ is the convex hull of the invariant distributions given the recurrent communication classes of $Q(m)$. 
		Therefore, the values of the map $s$ are non-empty, convex and compact.		
		Moreover, the graph of the map $s$ is closed: Let $(m^n, x^n)_{n \in \mathbb{N}}$ be a converging sequence such that $x^n \in s(m^n)$. Denote its limit by $(m,x)$. Then $0 = (x^n)^T Q(m^n)$ for all $n \in \mathbb{N}$. By continuity of $Q(\cdot)$ we have $0=x^TQ(m)$, which implies $x \in s(m)$.
		Thus, Kakutani's fixed point theorem yields a fixed point of the map $s$, which is an invariant distribution given $Q(\cdot)$.
	\end{proof}
	
	If $Q(m)$ is irreducible for all $m \in \mathcal{P}(\mathcal{S})$, the sets $s(m)$ will be singletons \cite[Theorem 4.2]{AsmussenQueues2003}.
	Let $x(m)$ denote this point.
	We remark that there are explicit representation formulas for $x(m)$ (e.g. \cite{NeumannComputation,ResnickAdventures}).
	With these insights we provide the following sufficient criterion for the uniqueness of the invariant distribution:
	
	\begin{theorem}
		Assume that $Q(m)$ is irreducible for all $m \in \mathcal{P}(\mathcal{S})$. Furthermore, assume that $f(m):=x(m)-m$ is continuously differentiable and that the matrix
		$$M(m) := \begin{pmatrix}
		\frac{\partial f_1(m)}{\partial m_1} & \ldots & \frac{\partial f_1(m)}{\partial m_{S-1}} \\
		\vdots & \ddots & \vdots \\
		\frac{\partial f_{S-1}(m)}{\partial m_1} & \ldots & \frac{\partial f_{S-1}(m)}{\partial m_{S-1}}
		\end{pmatrix} -
		\begin{pmatrix}
		\frac{\partial f_1(m)}{\partial m_S} & \ldots & \frac{\partial f_1(m)}{\partial m_S} \\
		\vdots & \ddots & \vdots \\
		\frac{\partial f_{S-1}(m)}{\partial m_S} & \ldots & \frac{\partial f_{S-1}(m)}{\partial m_S}
		\end{pmatrix}$$ is non-singular for all $m \in \mathcal{P}(\mathcal{S})$. Then there is a unique invariant distribution.
	\end{theorem}

	\begin{proof}
		We first note that any invariant distribution of a nonlinear Markov chain with generator $Q(\cdot)$ is an invariant distribution $m$ of a classical Markov chain with generator $Q(m)$.
		Since any invariant distribution of a classical Markov chain with generator $Q(m)$ has to satisfy that all components are strictly positive \cite[Theorem 4.2]{AsmussenQueues2003}, no invariant distribution of $Q(\cdot)$ lies on the boundary of $\mathcal{P}(\mathcal{S})$. 
		Therefore, we only need to ensure the existence of a unique invariant distribution in the interior of $\mathcal{P}(\mathcal{S})$.
		 
		The set $\mathcal{P}(\mathcal{S})$ is homeomorphic to $\bar{\Omega}$ with $$\Omega = \left\{m \in \mathbb{R}^{S-1}: m_i >0 \forall i \in \{1, \ldots, S-1\} \wedge \sum_{i =1}^{S-1} m_i <1 \right\},$$ where the continuous bijections are given as the restrictions of
		\begin{align*}
		&\phi: \mathbb{R}^{S-1} \rightarrow \mathbb{R}^S, \quad (m_1, \ldots, m_{S-1}) \mapsto \left(m_1, \ldots, m_{S-1}, 1- \sum_{i=1}^{S-1} m_i \right) \\
		&\psi: \mathbb{R}^S \rightarrow \mathbb{R}^{S-1}, \quad (m_1, \ldots, m_{S-1},m_S) \mapsto (m_1, \ldots, m_{S-1}).
		\end{align*}
		
		Define $\bar{f}: \bar{\Omega} \rightarrow \bar{\Omega}$ by $m \mapsto \psi(f(\phi(m)))$.
		By the chain rule we obtain
		\allowdisplaybreaks
		\begin{align*}
			\frac{\partial \bar{f}(m)}{\partial m} 
			&= \frac{\partial \psi}{\partial m}(f(\phi(m)) \cdot \frac{\partial f}{\partial m} (\phi(m)) \cdot \frac{\partial \phi}{\partial m} (m) \\
			&= \begin{pmatrix}
			1 & 0 & \ldots & 0 & 0 \\
			0 & 1 & \ddots & \vdots &\vdots \\
			\vdots & \ddots & \ddots & 0 & 0 \\
			0 & \ldots & 0 & 1 & 0 
			\end{pmatrix} \cdot
			\begin{pmatrix}
			\frac{\partial f_1(m)}{\partial m_1} & \ldots & \frac{f_1(m)}{\partial m_S} \\
			\vdots & \ddots & \vdots \\
			\frac{\partial f_S(m)}{\partial m_1} & \ldots & \frac{f_S(m)}{\partial m_S} 
			\end{pmatrix} \cdot
			\begin{pmatrix}
			1 & 0 & \ldots & 0 \\
			0 & 1 & \ddots & \ldots \\
			\vdots & \ddots &\ddots & 0 \\
			0 & \ldots & 0 & 1 \\
			-1 & -1 & \ldots & -1
			\end{pmatrix} \\
			&= M \left(\left(m_1, \ldots, m_{S-1}, 1- \sum_{i=1}^{S-1} m_i\right)^T\right).
		\end{align*} The matrix $M(m)$ is, by assumption, non-singular for all $m \in \mathcal{P}(\mathcal{S})$.
		Thus, 
		$$ \det \left( \frac{\partial \bar{f}(m)}{\partial m} \right) \neq 0 \quad \text{for all } m \in \bar{\Omega}.$$
		Since $\phi$, $\psi$, $f$ and $\text{det}$ are continuous functions, we obtain that also the function $m \mapsto \det ( \frac{\partial \bar{f}(m)}{\partial m})$ is continuous. Thus, the intermediate value theorem yields that $\det( \frac{\partial \bar{f}(m)}{\partial m})$ has uniform sign over $\bar{\Omega}$.
		
		Furthermore, we note that by assumption $M(m)$ is in particular non-singular for all $m \in \phi ( \bar{f}^{-1}(\{0\}))$.
		Thus, $0$ is a non-critical value of $\bar{f}$.

		The map $\bar{h}:[0,1] \times \bar{\Omega} \rightarrow \mathbb{R}^{S-1}$ given by
		\begin{align*}
			\bar{h}(t,m) &= t \cdot \bar{f}(m) + (1-t) \cdot \left( \frac{(S-1)}{S} (1, \ldots, 1)^T - m\right) \\
			&= t \cdot \psi (x(\phi(m)) + (1-t) \cdot \frac{S-1}{S} (1, \ldots, 1)^T -m
		\end{align*} is continuous.
		Furthermore, $0 \notin \bar{h}(t, \partial \Omega)$: Indeed, a point $m \in \partial \Omega$ either satisfies $m_i=0$ for some $i \in \{1, \ldots, S-1\}$ or $\sum_{i=1}^{S-1} m_i = 1$.
		However, by \cite[Theorem 4.2]{AsmussenQueues2003}, all components of the invariant distribution for an irreducible generator are strictly positive.
		Thus, we obtain in the first case that $h_i(t,m)>0$ and in the second case that the sum of all components is strictly negative, which in both cases implies that $h(t, m) \neq 0$.
		
		With these preparations we can make use of the Brouwer degree (see \cite[Section 1.1 and 1.2]{DeimlingDegree1985}), namely we obtain that
		$$\text{deg} \left( \frac{S-1}{S} (1, \ldots, 1)^T - m, \Omega, 0 \right) = \text{deg} (\bar{f}, \Omega, 0).$$
		Since for continuously differentiable maps $g$ and regular values $y \notin g(\partial \Omega)$ the degree is given by
		$$\text{deg}(g, \Omega, y) = \sum_{x \in g^{-1} (\{y\})} \text{sgn } \text{det} \left( \frac{\partial g}{\partial x}(x)\right),$$
		we obtain that
		$$(-1)^{S-1} = \sum_{m \in \bar{f}^{-1}(\{0\})} \text{sgn } \text{det} \left( \frac{\partial}{\partial m} \bar{f}(m) \right).$$
		Because the determinant has uniform sign over $\Omega \supseteq \bar{f}^{-1}(\{0\})$, we obtain that $\bar{f}^{-1}(\{0\})$ consists of exactly one element.
		Thus, there is a unique stationary point for the nonlinear Markov chain with nonlinear generator $Q(\cdot)$.		
	\end{proof}

	\begin{example}
		We illustrate the use of the result in an example: Consider a nonlinear Markov chain with the following generator
		\[
		Q(m) = \begin{pmatrix}
		-(b+ em_1+ \epsilon) & b & em_1 + \epsilon \\
		0 & -(em_2 + \epsilon) & em_2 + \epsilon \\
		\lambda & \lambda & - 2 \lambda
		\end{pmatrix},
		\]
		where all constants are strictly positive. This nonlinear Markov chain arises in a mean field game model of consumer choice with congestion effects (see \cite{NeumannDiss}, also for detailed calculations). In this setting the invariant distributions are given as the solution(s) of the nonlinear equation $0 = m^T Q(m)$, for which closed form solutions are hard or impossible to obtain. However, it is possible to verify that the matrix $M(m)$ is non-singular for all $m \in \mathcal{P}(\mathcal{S})$ yielding a unique invariant distribution. This information can in particular be used, to obtain certain characteristic properties of the solutions.
	\end{example}

	\section{Examples for Peculiar Limit Behaviour}
	\label{Section:Examples}
	
	The following examples show that the limit behaviour for nonlinear Markov chains (also in the case of small state spaces) is more complex than for classical continuous time Markov chains.
	In particular, the marginal distributions might not converge, but are periodic and a nonlinear Markov chain with an irreducible nonlinear generator might not be strongly ergodic, but we observe convergence towards several different invariant distributions.
		
	\subsection{An Example with Periodic Marginal Distributions}
	\label{Subsection:Periodic}
	
	Let $B = \mathcal{P}(\{1,2,3\}) \cap \{m \in \mathbb{R}^3: \min \{m_1, m_2, m_3\}  \ge \frac{1}{10} \}$ and set for all $m \in B$ the matrix $Q$ as follows
	\begin{align*}
		Q_{13}(m) &= 
		\frac{1}{m_3} \left( \frac{1}{3} - m_1 \right) \mathbb{I}_{\{m_1 \le \frac{1}{3} \}} &
		Q_{23}(m) &= \frac{1}{m_1} \left( m_1 - \frac{1}{3} \right) \mathbb{I}_{\{m_1 \ge \frac{1}{3} \}} \\
		Q_{31}(m) &= \frac{1}{m_3} \left( m_2 - \frac{1}{3} \right) \mathbb{I}_{\{m_2 \ge \frac{1}{3} \}} & 
		Q_{32}(m) &= \frac{1}{m_2} \left( \frac{1}{3} - m_2 \right) \mathbb{I}_{\{m_2 \le \frac{1}{3} \}} \\
		Q_{12} (m) &= Q_{21}(m) = 0 & Q_{ii}(m) &= - \sum_{j \neq i} Q_{ij}(m), 
	\end{align*} where $\mathbb{I}_A$ is $1$ if $A$ is true and $0$ else.
	Since all transition rates on $B$ are Lipschitz continuous functions, there is an extension of $Q_{ij}(\cdot)$ on $\mathcal{P}(\mathcal{S})$ for all $i, j \in \mathcal{S}$, which is again Lipschitz continuous.
	Thus, a nonlinear Markov chain with generator $Q$ exists.
	The ordinary differential equation characterizing the marginals on $B$ reads
	\begin{align*}
	\frac{\partial}{\partial t} \Phi_1^t(m_0) &= \begin{cases}
	\Phi_1^t(m_0) \cdot \left( - \frac{1}{\Phi_1^t(m_0)} \left(\frac{1}{3} - \Phi_2^t(m_0) \right) \right) & \Phi_2^t(m_0) \le \frac{1}{3} \\
	\Phi_3^t(m_0) \cdot \left( \frac{1}{\Phi_3^t(m_0)} \left( \Phi_2^t(m_0) - \frac{1}{3} \right) \right) & \Phi_2^t(m_0) \ge \frac{1}{3}
	\end{cases} \\
	&= \Phi_2^t(m_0) - \frac{1}{3} \\
	\frac{\partial}{\partial t} \Phi_2^t(m_0) &= \begin{cases}
	\Phi_2^t(m_0) \cdot \left( - \frac{1}{\Phi_2^t(m_0)} \left( \Phi_1^t(m_0) - \frac{1}{3} \right) \right) & \Phi_1^t(m_0) \ge \frac{1}{3} \\
	\Phi_3^t(m_0) \cdot \left( \frac{1}{\Phi_3^t(m_0)} \left( \frac{1}{3} - \Phi_1^t(m_0) \right) \right) & \Phi_1^t(m_0) \le \frac{1}{3} 
	\end{cases} \\
	&= \frac{1}{3} - \Phi_1^t(m_0) \\
	\frac{\partial}{\partial t} \Phi_3^t(m_0) &= \Phi_1^t(m_0) -\Phi_2^t(m_0).	
	\end{align*}
	Thus, for any neighbourhood $U \subseteq B$ of $\left( \frac{1}{3}, \frac{1}{3}, \frac{1}{3} \right)^T$ the first two components of the marginal behave like the classical harmonic oscillator.
	Therefore, there are initial distributions such that the marginals are periodic. 
	An example is the initial distribution $m_0 = (0.2, 0.4, 0.4)$ for which the marginals are plotted in Figure \ref{NonlinearCTMCPEriodicTrajectory}.
	
	\begin{figure}[h]
		\begin{center}
			\includegraphics[scale=0.8]{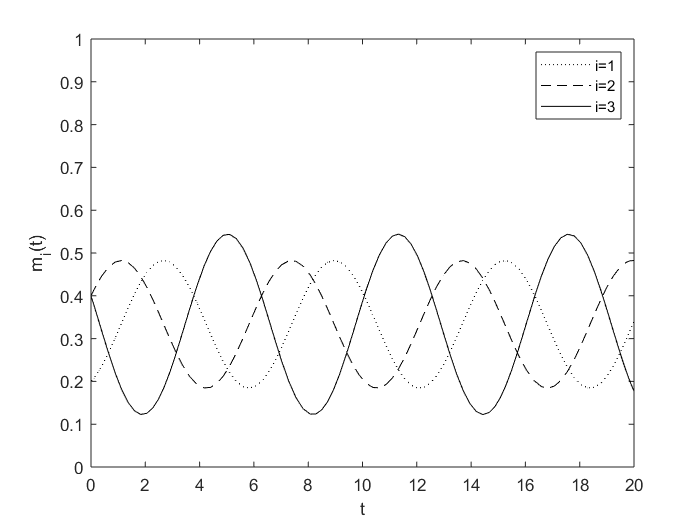}
		\end{center}
		\caption{The marginal distributions of the nonlinear continuous time Markov chain with initial distribution $m_0 = (0.2,0.4,0.4)$.}
		\label{NonlinearCTMCPEriodicTrajectory}
	\end{figure}
	
	\subsection{An Example of a Nonlinear Markov Chain with Irreducible Generator that is not Strongly Ergodic}
	
	Let 
	$$Q(m) = \begin{pmatrix}
	- \left( \frac{29}{3} m_1^2 - 16 m_1 + \frac{22}{3}\right) &
	\frac{29}{3} m_1^2 - 16 m_1 + \frac{22}{3} \\
	m_1^2 + m_1 +1  & - \left( m_1^2 + m_1 +1 \right) 
	\end{pmatrix}.$$
	This matrix is irreducible for all $m \in \mathcal{P}(\{1,2\})$ since $m_1^2 + m_1 +1 \ge 1$ and $\frac{29}{3} m_1^2 - 16 m_1 + \frac{22}{3} \ge \frac{62}{82}$ for all $m_1 \ge 0$.
	
	The ordinary differential equation describing the marginals for the initial condition $m_0 \in \mathcal{P}(\{1,2\})$ is given by
	\begin{align*}
		\frac{\partial}{\partial t} \Phi_1^t(m_0) &= - \frac{32}{3} \left(\Phi_1^t(m_0)\right)^3 + 16 \left(\Phi_1^t(m_0)\right)^2 - \frac{22}{3} \Phi_1^t(m_0) +1=:f \left(\Phi_1^t(m_0)\right)\\
		\frac{\partial}{\partial t} \Phi_2^t(m_0) &= \frac{32}{3} \left(\Phi_1^t(m_0)\right)^3 - 16 \left(\Phi_1^t(m_0)\right)^2 + \frac{22}{3} \Phi_1^t(m_0) -1 = -f \left(\Phi_1^t(m_0)\right).
	\end{align*} We obtain that there are three stationary points $m^1= (0.25,0.75)$, $m^2 = (0.5,0.5)$ and $m^3=(0.75,0.25)$ and the following convergence behaviour:
	\begin{itemize}
		\item Since the function $f(\cdot)$ is strictly positive on $[0,0.25)$, the trajectories will for all initial conditions $(m_0)_1 \in [0,0.25)$ converge towards $m_1 = 0.25$. 
		\item Since the function $f(\cdot)$ is strictly negative on $(0.25,0.5)$, the trajectories will for all initial conditions $(m_0)_1 \in (0.25,0.5)$ converge towards $m_1 = 0.25$. 
		\item Since the function $f(\cdot)$ is strictly positive on $(0.5,0.75)$, the trajectories will for all initial conditions $(m_0)_1 \in (0.5,0.75)$ converge towards $m_1 = 0.75$. 
		\item Since the function $f(\cdot)$ is strictly negative on $(0.75,1]$, the trajectories will for all initial conditions $(m_0)_1 \in (0.75,1]$ converge towards $m_1 = 0.75$. 
	\end{itemize}
	This behaviour is visualized in Figure \ref{IrreducibleNotStrongErgodicFigure}, where several trajectories for different initial conditions are plotted.
	
	\begin{figure}[h]
		\begin{center}
			\includegraphics[scale=0.8]{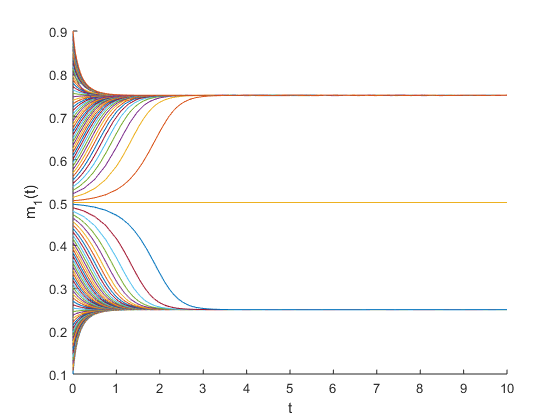}
		\end{center}
		\caption{The trajectories of the nonlinear Markov chain for several initial conditions.}
		\label{IrreducibleNotStrongErgodicFigure}
	\end{figure}
 	
	\section{Sufficient Criteria for Ergodicity for Small State Spaces}
	\label{Section:Ergodicity}
	
	Although the limit behaviour is more complex for nonlinear Markov chains, we still obtain sufficient criteria for ergodicity in the case of a small number of states. Here, we present these criteria, discuss applicability as well as the problems that occur for larger state spaces.
	
	\begin{theorem}
		\label{MyopicGlobal:NonlinearMC:StrongErgodicS=2}
		Let $S=2$ and assume that $f: [0,1] \rightarrow \mathbb{R}$ defined via $$f(m_1) := m_1 \cdot (Q_{11}(m_1,1-m_1)) + (1-m_1) \cdot Q_{21}(m_1,1-m_1)$$ is continuous. Furthermore, assume that $(\bar{m},1-\bar{m})$ is the unique stationary point given $Q$. 
		Then, the nonlinear Markov chain is strongly ergodic.
	\end{theorem}
	
	\begin{proof}
		An equilibrium point is characterized by the property that $\frac{\partial}{\partial t} \Phi^t(m)=0$.
		By flow invariance of $\mathcal{P}(\mathcal{S})$ for the ordinary differential equation $\frac{\partial}{\partial t} \Phi^t(m_0) = \Phi^t(m_0) Q(\Phi^t(m_0))$ (see the proof of Theorem \ref{LipschitzContinuousGenerator}), which implies that $\frac{\partial}{\partial t} \Phi^t_1(m) + \frac{\partial}{\partial t} \Phi^t_2(m) =0$, this property is equivalent to the fact that $\frac{\partial}{\partial t} \Phi^t_1(m) = 0$.
		
		Since $\frac{\partial}{\partial t} \Phi^t_1(m) = f(m_1)$ and since we have a unique equilibrium point, we obtain that $f(\bar{m})=0$ and $f(m_1) \neq 0$ for all $m_1 \neq \bar{m}$.
		Since $f(\cdot)$ is continuous, we obtain that $f(\cdot)$ is non-vanishing on $[0,\bar{m})$ and $(\bar{m}, 1]$ and has uniform sign on each of these sets. 
		Since $Q(\cdot)$ is a conservative generator we moreover obtain that $f(0)\ge 0$ and $f(1)\le 0$. 
		Thus,  we obtain that $f(m_1)>0$ for all $m_1 \in [0,\bar{m})$ and $f(m_1)<0$ for all $m_1 \in (\bar{m},1]$.
		This in turn yields that $[0,1]$ is flow invariant for $\dot{m}_1 = f(m_1)$.
		
		Fix $m_0 \in \mathcal{P}(\mathcal{S})$. Then the systems $\Phi^t(m_0) = Q(\Phi^t(m_0))^T \Phi^t(m_0)$ and $\tilde{\Phi}^t(m_0)_1 = f(\tilde{\Phi}^t(m_0))$ are equivalent in the sense that $\Phi^t_1(m_0)= \tilde{\Phi}^t(m_0)$ for all $t \ge 0$, $m_0 \in \mathcal{P}(\{1,2\})$:
		Indeed, let $\Phi^t(m_0) = (\Phi^t_1(m_0),\Phi^t_2(m_0))$ be a solution of the differential equation $\frac{\partial}{\partial t} \Phi^t(m_0) = Q(\Phi^t(m_0))^T \Phi^t(m_0)$ with initial condition $\Phi^0(m_0)=m_0$. 
		By flow invariance of $\mathcal{P}(\mathcal{S})$ for $\frac{\partial}{\partial t} \Phi^t(m_0) = Q(\Phi^t(m_0))^T \Phi^t(m_0)$ (see Theorem \ref{LipschitzContinuousGenerator}), we have $\Phi^t_2(m_0) = 1-\Phi^t_1(m_0)$ for all $t \ge 0$. 
		Thus, $\frac{\partial}{\partial t} \Phi^t(m_0) = Q(\Phi^t(m_0))^T \Phi^t(m_0)$ is equivalent to
		\begin{align}
		\label{MyopicGlobal:Nonlinear:Equivalence}
		\begin{cases}
		\frac{\partial}{\partial t} \Phi^t_1(m_0) &=  \Phi^t_1(m_0) \cdot (Q_{11}(\Phi^t_1(m_0),1-\Phi^t_1(m_0))) \\
		&\quad  + (1-\Phi^t_1(m_0)) \cdot Q_{21}(\Phi^t_1(m_0),1-\Phi^t_1(m_0)) \\
		-\frac{\partial}{\partial t} \Phi^t_1(m_0) &=  \Phi^t_1(m_0) \cdot (-Q_{12}(\Phi^t_1(m_0),1-\Phi^t_1(m_0))) \\
		&\quad  + (1-\Phi^t_1(m_0)) \cdot Q_{22}(\Phi^t_1(m_0),1-\Phi^t_1(m_0)).
		\end{cases}
		\end{align} 
		Therefore, $\Phi^t_1(m_0)$ is indeed a solution of $\Phi^t_1(m_0) = f(\Phi^t_1(m_0))$.
		For the converse implication we first note that because $Q(m)$ is conservative for all $m \in \mathcal{P}(\mathcal{S})$ the last equation of \eqref{MyopicGlobal:Nonlinear:Equivalence} is the first equation multiplied by $(-1)$ .
		If $\tilde{\Phi}^t(m_0)$ satisfies $\frac{\partial}{\partial t} \tilde{\Phi}^t(m_0) = f(\tilde{\Phi}^t(m_0))$, $\tilde{\Phi}^0(m_0) = (m_0)_1 \in [0,1]$, then, by flow invariance, $\tilde{\Phi}^t(m_0)) \in [0,1]$ for all $t\ge0$.
		Thus, the function $\Phi^t(m_0) = (\tilde{\Phi}^t(m_0), 1- \tilde{\Phi}^t(m_0))$ satisfies $\frac{\partial}{\partial t} \Phi^t(m_0) = Q(\Phi^t(m_0))^T \Phi^t(m_0)$. 	
		
		The desired convergence statement directly follows from $f(m_1)>0$ for all $m_1 \in [0,\bar{m})$ and $f(m_1)<0$ for all $m_1 \in (\bar{m},1]$.
	\end{proof}

	We also obtain a sufficient criterion for the case of three states. The proof technique is similar to the two state case. Indeed, we first show that our system is equivalent to a two-dimensional system, for which we can then use standard tools for two-dimensional dynamical systems exploiting that the dynamical system has a particular shape since $Q(\cdot)$ is a conservative generator.

	As mentioned, we obtain for systems with three states that given $m_0 \in \mathcal{P}(\mathcal{S})$ the function $\Phi^t(m_0) = (\Phi^t_1(m_0), \Phi^t_2(m_0), \Phi^t_3(m_0))$ is a solution of $\frac{\partial}{\partial t} \Phi^t(m_0) =Q(\Phi^t(m_0))^T \Phi^t(m_0)$, $\Phi^0(m_0)=m_0$ if and only of $(\Phi^t_1(m_0), \Phi^t_2(m_0))$ is a solution of $$\begin{pmatrix}
	\frac{\partial}{\partial t} \Phi^t_1(m_0)\\ \frac{\partial}{\partial t} \Phi^t_2(m_0)
	\end{pmatrix} = f \begin{pmatrix}
	\Phi_1^t(m_0) \\ \Phi_2^t(m_0)
	\end{pmatrix}, \quad \begin{pmatrix}
	\Phi^0_1(m_0)) \\ \Phi^0_2(m_0)
	\end{pmatrix} = \begin{pmatrix}
	(m_0)_1 \\ (m_0)_2
	\end{pmatrix},$$ where 
	\begin{equation}
	\label{NonlinearMC:Sufficient:DefinitionF}
	f \begin{pmatrix}
	m_1 \\ m_2
	\end{pmatrix} = \begin{pmatrix}
	Q_{31}(\hat{m}) + (Q_{11} (\hat{m}) - Q_{31}(\hat{m})) m_1  + (Q_{21}(\hat{m})- Q_{31}(\hat{m})) m_2  \\
	Q_{32}(\hat{m}) + (Q_{12}(\hat{m}) - Q_{32}(\hat{m})) m_1 + (Q_{22}(\hat{m}) - Q_{32}(\hat{m})) m_2
	\end{pmatrix}	
	\end{equation} and $\hat{m} = (m_1,m_2, 1- m_1-m_2)$.
	Indeed, the proof is analogous to the proof for the two state case, the central  adjustment is to prove the flow invariance of $\{(m_1,m_2) \in [0,\infty): m_1 + m_2 \le 1\}$ for $(\Phi^t_1(m_0), \Phi^t_2(m_0))^T = f(\Phi^t_1(m_0),\Phi^t_2(m_0))$ instead of the flow invariance of $[0,1]$ for $\Phi^t_1(m_0) = f(\Phi^t_1(m_0))$. This statement is proven in the appendix (Lemma \ref{Lemma:Appendix}).
	
	To show the desired convergence statement, we now rely on the Poincar\'{e}-Bendixson Theorem \cite[Chapter 7]{TeschlODE}, which characterizes the $\omega$-limit sets $\omega_+(m_0)$ of a trajectory with initial condition $\Phi^0(m_0)=m_0$:
	
	\begin{theorem}
		\label{Theorem:ErgodicityS3}
		Let $O \supseteq \{(m_1,m_2) \in [0,\infty)^2 : m_1 + m_2 \le 1\}$ be a simply connected and bounded region such that there is a continuously differentiable function $f:O \rightarrow \mathbb{R}^2$  satisfying \eqref{NonlinearMC:Sufficient:DefinitionF} on $\mathcal{P}(\mathcal{S})$. 
		Let $\bar{m}$ be the unique stationary point given $Q(\cdot)$. 
		Furthermore, assume that
		\begin{itemize}
			\item[(a)] $\frac{\partial f_1}{\partial m_1} (m) + \frac{\partial f_2}{\partial m_2} (m)$ is non-vanishing for all $m \in O$ and has uniform sign  on $O$,
			\item[(b)] it holds that
			$$\frac{\partial f_1}{\partial m_1} (\bar{m}) \cdot \frac{\partial f_2}{\partial m_2} (\bar{m}) - \frac{\partial f_1}{\partial m_2} (\bar{m}) \cdot \frac{\partial f_2}{\partial m_1} (\bar{m})>0$$ or it holds that
			$$\left(\frac{\partial f_1}{\partial m_1} (\bar{m}) + \frac{\partial f_2}{\partial m_2} (\bar{m}) \right)^2 - 4 \left( \frac{\partial f_1}{\partial m_1} (\bar{m}) \cdot \frac{\partial f_2}{\partial m_2} (\bar{m}) - \frac{\partial f_1}{\partial m_2} (\bar{m}) \cdot \frac{\partial f_2}{\partial m_1} (\bar{m}) \right) <0 .$$
		\end{itemize}
		Then, the nonlinear Markov chain is strongly ergodic.
	\end{theorem}

	\begin{proof}
		Since  the set $F:= \{ (m_1, m_2)^T \in \mathbb{R}^2: m_1, m_2 \ge 0 \wedge m_1 + m_2 \le 1\}$ is flow invariant for $(\frac{\partial}{\partial t} \Phi^t_1(m_0),\frac{\partial}{\partial t} \Phi^t_2(m_0))^T = f(\Phi^t_1(m_0),\Phi^t_2(m_0))$, any trajectory will stay in this set.
		Since the set $F$ is compact, we obtain by \cite[Lemma 6.6]{TeschlODE} that $\omega_+(m_0 )$ lies $F$.
		Since there is, by assumption, only one stationary point we can apply the Poincar\'{e}-Bendixson Theorem \cite[Theorem 7.16]{TeschlODE}. 
		It yields that one of the following three cases holds:
		\begin{itemize}
			\item[(i)] $\omega_+(m_0) = \{\bar{m}\}$
			\item[(ii)] $\omega_+(m_0)$ is a regular periodic orbit
			\item[(iii)] $\omega_+(m_0)$ consists of (finitely many) fixed points $x_1, \ldots, x_k$ and non-closed orbits $\gamma(z)$ such that $\omega_\pm(z) \in \{x_1, \ldots, x_k\}$.
		\end{itemize}
		By condition (a) and Bedixson's criterion \cite[Theorem 3.5]{JordanNonlinear} the case (ii) is not possible.
		Since, by condition (b), the point $\bar{m}$ is not a saddle point, there is no homoclinic path joining $\bar{m}$ to itself.
		Therefore, since $\bar{m}$ is the only stationary point, also case (iii) is not possible.
		Thus, $\omega_+(m_0)=\{\bar{m}\}$.
		Since the considered trajectory lies in the compact set $F$, we moreover obtain by \cite[Lemma 6.7]{TeschlODE}
		that
		$$0 = \lim_{t \rightarrow \infty} d\left(\Phi^t(m_0), \omega_+(m_0) \right) = \lim_{t \rightarrow \infty} d\left(\Phi^t(m_0), \bar{m} \right).$$
	\end{proof}
	
	\begin{remark}
		The equivalence of the considered systems and $S-1$ systems on  some subset of $\mathbb{R}^{S-1}$ as well as the construction performed in Section \ref{Subsection:Periodic} hint the general problem for a larger number of states ($S \ge 4$). 
		It might happen that the dynamics of the nonlinear Markov chain describe a classical ``chaotic'' nonlinear system like the Lorentz system. 
		In other words, the difficulties that arise in the classical theory of dynamical systems might also arise here, for which reason criteria for a larger number of states are more complex. 
	\end{remark}

	\begin{example}
		Theorem \ref{Theorem:ErgodicityS3} now yields strong ergodicity of the nonlinear Markov chain introduced in the end of Section \ref{Section:InvariantDist}. In this setting the function $f$ is given by
		\[
		f \begin{pmatrix}
		m_1 \\ m_2
		\end{pmatrix} =
		\begin{pmatrix}
		\lambda - e m_1^2 - (b + \epsilon + \lambda) m_1 - \lambda m_2 \\
		\lambda + (b- \lambda) m_1 - e m_2^2 - (\epsilon + \lambda) m_2
		\end{pmatrix}
		\] and we moreover have $\frac{\partial f_1}{\partial m_1} (m) + \frac{\partial f_2}{\partial m_2}(m) <0$ for all $m \in N_\epsilon([0,1]^2)$ as well as 
		\[
		\frac{\partial f_1}{\partial m_1} (m) \frac{\partial f_2}{\partial m_2} (m) - \frac{\partial f_1}{\partial m_2}(m) \frac{\partial f_2}{\partial m_1} (m) >0
		\] for all $m \in [0,1]^2$ and, thus, in particular for the unique invariant distribution.
		Therefore, by Theorem \ref{Theorem:ErgodicityS3} we obtain strong ergodicity.
	\end{example}

	\appendix
	\section{Appendix}
	\label{appendix}
	
	\begin{proof}[Proof of Theorem \ref{LipschitzContinuousGenerator}]
		We first note that
		$$f(m):= \left( \sum_{i \in \mathcal{S}} m_i Q_{ij}(m) \right)_{j \in \mathcal{S}}$$ is Lipschitz continuous on $\mathcal{P}(\mathcal{S})$:
		Indeed, let $L$ be a Lipschitz constant for all functions $Q_{ij}(\cdot)$ ($i,j \in \mathcal{S}$) simultaneously. 
		Moreover, since $\mathcal{P}(\mathcal{S})$ is compact there is a finite constant $$M:= \sup_{m \in \mathcal{P}(\mathcal{S}), i,j \in \mathcal{S}} Q_{ij}(m).$$
		Thus, we have
		\allowdisplaybreaks
		\begin{align*}
			|f(m^1) - f(m^2)|_1 &\le (M+L)S \cdot \left| m^1 - m^2\right|_1.
		\end{align*}
		By McShane's extension theorem \cite{McShaneExtensionTheorem} there is a Lipschitz continuous extension $\tilde{f}: \mathbb{R}^S \rightarrow \mathbb{R}^S$ of $f$. 
		Let us fix an arbitrary $m_0 \in \mathcal{P}(\mathcal{S})$.
		By the classical existence and uniqueness theorem for ordinary differential equations, we obtain that there is a unique solution of $\Phi^\cdot(m_0) :[0,\infty) \rightarrow \mathbb{R}^S$ of $\frac{\partial}{\partial t} \Phi^t(m_0) = \tilde{f}(\Phi^t(m_0)), \Phi^0(m_0) = m_0$.
		
		As a next step we show that the vectors $f(m) = \tilde{f}(m)$ lie for all $m \in \mathcal{P}(\mathcal{S})$ in the Bouligand tangent cone 
		\begin{align*}
			T_{\mathcal{P}(\mathcal{S})}(m) &= \left\{ y \in \mathbb{R}^S: \liminf_{h \downarrow 0} \frac{d(m+hy, \mathcal{P}(\mathcal{S}))}{h} = 0 \right\} \\
			&= \left\{ y \in \mathbb{R}^S : y_i \ge 0 \forall i \in \mathcal{S} \text{ s.t. } m_i = 0 \wedge \sum_{i \in \mathcal{S}} y_i = 0 \right\},
		\end{align*} where the second line follows from \cite[Proposition 5.1.7]{AubinDifferentialInclusions}:		
		Indeed, since for all interior points of $\mathcal{P}(\mathcal{S})$ the condition is trivially satisfied, it suffices to consider the boundary points $m \in \partial \mathcal{P}(\mathcal{S})$.
		These points satisfy that there is at least one $j \in \mathcal{P}(\mathcal{S})$ such that $m_j=0$.
		Since the only non-positive column entry of $Q_{\cdot j}$ (which is $Q_{jj}$) gets weight $m_j$, the vector $f(m) = (\sum_{i \in \mathcal{S}} m_i Q_{ija}(m))_{j \in \mathcal{S}}$ will have non-negative entries at each $j \in \mathcal{S}$ such that $m_j = 0$.
		Since $Q$ is conservative, we moreover obtain that 
		$$\sum_{j \in \mathcal{S}} \sum_{i \in \mathcal{S}} m_i Q_{ija}(m) = \sum_{i \in \mathcal{S}} \underbrace{\sum_{j \in \mathcal{S}} Q_{ija}(m)}_{=0} m_i = 0.$$		
		Thus, $f(m) = \tilde{f}(m) \in T_{\mathcal{P}(\mathcal{S})}(m)$ for all $m \in \mathcal{P}(\mathcal{S})$. 
		Therefore, we obtain, by the classical flow invariance statement for ordinary differential equations (\cite[Theorem 10.XVI]{WalterODE1998}), that the solution satisfies $m(t) \in \mathcal{P}(\mathcal{S})$ for all $t \ge 0$.
		Thus, $\Phi^\cdot(m_0) :[0,\infty) \rightarrow \mathbb{R}^S$ is also the unique solution of $\frac{\partial}{\partial t} \Phi^t(m_0) = f(\Phi^t(m_0)), \Phi^0(m_0) = m_0$.
		The continuity of $\Phi^t(\cdot)$ follows from a classical general dependence theorem \cite[Theorem 12.VII]{WalterODE1998}.		
		\end{proof}
	
		\begin{lemma}
			\label{Lemma:Appendix}
			The set $N= \{(m_1,m_2) \in [0,\infty): m_1 + m_2 \le 1 \}$ is flow invariant for $(\Phi^t_1(m_0), \Phi^t_2(m_0))^T = f(\Phi^t_1(m_0),\Phi^t_2(m_0))$.
		\end{lemma}
	
	\begin{proof}
		The statement follows from \cite[Lemma 1]{FernandesPracticalSets}.
		This lemma states that for an open set $O \subseteq \mathbb{R}^S$ and a family of continuously differentiable functions $g_i:O \rightarrow \mathbb{R}$ ($i\in \{1, \ldots, k\}$) the set 
		\[
		M = \{x \in O: g_i(x) \le 0 \text{ for all } i \in \{1, \ldots, k\} \}
		\] is flow invariant for $\dot{x} = f(x)$ whenever for any $x \in \partial M$ there is an $i \in \{1, \ldots, k\}$ such that $g_i(x)=0$ and 
		\[ 
		\langle f(x), \nabla g_i(x) \rangle <0.
		\]
		Indeed, in our case we have
		\[
		M = \{ x \in \mathbb{R}^S: -m_1 \le 0 \wedge -m_2 \le 0 \wedge m_1 + m_2 \le 1\}
		\]
		and the	boundary points of this set either satisfy $m_i=0$ for at least one $i \in \{1,2\}$ or $m_1+m_2 =1$.
		Since $Q(\cdot)$ is conservative and irreducible, we obtain 
		$		\left\langle f ((m_1,m_2)^T) , \nabla (-m_i) \right\rangle <0$ in the first case and 
		$		\left\langle f ((m_1,m_2)^T), \nabla (m_1 + m_2 -1) \right\rangle <0$ in the second case.
		Therefore, the claim follows.
	\end{proof}

	\bibliographystyle{abbrvnat}
	\bibliography{literatureNonlinearMC}

\end{document}